\documentclass{amsart}
\usepackage[utf8]{inputenc}
\usepackage{amsmath}
\usepackage{amssymb}
\usepackage{caption}
\usepackage{amsthm}

\usepackage[dvipsnames]{xcolor}
\usepackage[backref=false, colorlinks, linkcolor=blue, citecolor=Blue]{hyperref}

\usepackage{nicefrac}
\usepackage[inline]{enumitem}
\usepackage{enumitem}

\usepackage{todonotes}

\usepackage{amssymb}
\usepackage{mathrsfs}

\usepackage{array,float}

\usepackage{tikz}
\usetikzlibrary{shapes,arrows}
\usetikzlibrary{fit,positioning}
\usetikzlibrary{patterns,decorations.pathreplacing}
\usepackage{xkeyval}
\usepackage{moreverb}
\usepackage{epic}
\usepgfmodule{shapes,plot,decorations}

\usepackage{booktabs} 
\usepackage[normalem]{ulem}

\newcommand{\Z}{\mathbb{Z}}

\newcommand{\R}{\mathbb{R}}

\newcommand{\HH}{\mathbb{H}}

\newcommand{\Zcl}{\mathrm{Zcl}}

\newcommand{\PO}{\mathbf{PO}}

\usepackage[warn]{mathtext}

\renewcommand{\P}{\mathscr{P}}

\newcommand{\SSS}{\mathbb{S}}

\newcommand{\Sym}{\textrm{Sym}}

\newcommand{\Isom}{\mathrm{Isom}}

\newcommand{\OO}{\mathrm{O}}

\newtheorem{theorem}{Theorem}[section]
\newtheorem{corollary}{Corollary}[theorem]
\newtheorem{lemma}[theorem]{Lemma}

\newtheorem{proposition}[theorem]{Proposition}

\theoremstyle{remark}
\newtheorem{remark}[theorem]{Remark}

\theoremstyle{remark}

\title{Thin hyperbolic reflection groups}

\author{Nikolay Bogachev}
\address{Department of Computer and Mathematical Sciences, University of Toronto Scarborough, 1095 Military Trail, Toronto ON, M1C 1A3 Canada}
\email{n.bogachev@utoronto.ca}

\author{Alexander Kolpakov}
\address{Institut de Math\'ematiques, Universit\'e de Neuch\^atel, Rue Emile--Argand 11, CH--2000 Neuch\^atel, Suisse / Switzerland}
\email{kolpakov.alexander@gmail.com}

\begin{document}

\begin{abstract}
We study a family of Zariski dense finitely generated discrete subgroups of $\Isom(\HH^d)$, $d \geqslant 2$, defined by the following property: any group in this family contains at least one reflection in a hyperplane. As an application we obtain a general description of all thin hyperbolic reflection groups. In particular, we show that the Vinberg algorithm applied to a non--reflective Lorentzian lattice gives rise to an infinite sequence of thin reflection subgroups in $\Isom(\HH^d)$, for any $d \geqslant 2$. Moreover, every such group is a subgroup of a group produced by the Vinberg algorithm applied to a Lorentzian lattice independently on the latter being reflective. As a consequence, all thin hyperbolic reflection groups are enumerable. 
\end{abstract}

\subjclass[2020]{}


\maketitle

\section{Introduction}

Let $\HH^d$ be the real hyperbolic space of dimension $d$. Then its full isometry group $\Isom(\HH^d)$ is naturally embedded in $\OO_{d,1}(\R)$ as a subgroup of index $2$. Here and below, by a {\em (hyperbolic) reflection group} we mean a discrete subgroup of $\Isom(\HH^d)$ generated by reflections in hyperplanes. 

In this note, we provide a characterization of Zariski dense finitely generated discrete subgroups of $\Isom(\HH^d)$ for $d \geqslant 2$ that contain reflections in hyperplanes; see Theorem~\ref{th:main}. Out of curiosity, we provide two different proofs, one of an algebraic nature and the other being geometric. Moreover, Theorem~\ref{th:main} is apllied in order to describe {\em thin} reflection subgroups of arithmetic hyperbolic lattices; see Theorem~\ref{th:thin} and Theorem~\ref{th:thin-2}.

\begin{theorem}\label{th:main}
    Let $\Gamma < \Isom(\HH^d)$ be a finitely generated Zariski dense discrete group containing at least one reflection. Then $\Gamma$ contains a discrete Zariski dense subgroup generated by finitely many reflections. Moreover, if the maximal reflection subgroup of $\Gamma$ is infinite--index then $\Gamma$ contains a sequence of subgroups $W_n$ all being Zariski dense finitely generated hyperbolic reflection groups such that $W_n$ is of infinite index in $W_{n+1}$ for all $n$. 
\end{theorem}

\begin{remark}\label{rem-1}
Despite of having a general algebraic proof of Theorem~\ref{th:main}, we also provide a geometric justification for this result. This geometric proof relies on hyperbolic geometry, and can be useful for the reader to get deeper understanding of geometric behaviour of Zariski dense hyperbolic reflection groups. We would like to stress the fact that the algebraic proof can be potentially generalized to the more general setting of subgroups of $\OO_{p,q}(\R)$; see Remark \ref{rem:vin-generalization} for more details. 
\end{remark}

The discussion of Zariski dense finitely generated groups is interesting for us in particular in the context of the so--called {\em thin} groups. 

Let $G$ be a semi--simple Lie group and $\Gamma < G$ be its arithmetic subgroup with ambient algebraic group $\mathbf{G}$ and field of definition $k$. Assume for simplicity that $\Gamma$ is a finite--index subgroup of $\mathbf{G}(\mathcal{O}_k)$. A group $\Lambda < \Gamma$ is called \textit{thin} if
\begin{itemize}
    \item[a)] $\Lambda$ is finitely generated,
    \item[b)] $\Lambda$ has infinite index in $\Gamma$, i.e. $\mu(G/\Lambda) = \infty$ w.r.t. the Haar measure $\mu$ on $G$,
    \item[c)] $\Lambda$ is Zariski dense in $\mathbf{G}(\R)$, which is isogenous to $G^\circ$.
\end{itemize}

Recently thin groups, usually considered as subgroups of $\mathrm{SL}_n(\Z)$, with $n\geq 2$, attracted attention in regard to their number--theoretic and dynamical properties, as well as because of their relation to the monodromy of hypergeometric functions. The reader may consult \cite{KLLR19, Sarnak} for more information on thin groups. 

Let $k$ be a totally real number field with the ring of integers $\mathcal{O}_k$. Let $L$ be a {\em Lorentzian lattice} of signature $(d,1)$, see Section~\ref{subsec:Lor} for precise definitions. Then the hyperbolic space $\HH^d$ is naturally embedded in $L \otimes_{\mathrm{id}(\mathcal{O}_k)}\mathbb{R} \cong \R^{d,1}$. Let $\mathcal{O}'(L)$ be the group of linear transformations with $\mathcal{O}_k$--coefficients that preserve both $L$ and $\HH^d$.

It is known that $\Gamma = \mathbf{PO}(L)_{\mathcal{O}_k} \cong \mathcal{O}'(L)$ is an arithmetic lattice in the Lie group $G = \PO(L)_\R \cong \Isom(\HH^d)$. Thus, the action of $\Gamma$ on $\HH^d$ has a finite--volume fundamental domain. Lattices in $G$ commensurable with $\Gamma$ as constructed above are called \textit{arithmetic hyperbolic lattices of simplest type}. This is the only type of arithmetic groups in $\Isom\,\HH^d$ that may contain reflections in hyperplanes: see Vinberg \cite[Lemma 7]{Vin67}. 

Given a Lorentzian lattice $L$, let $\mathscr{R}(L)$ be the subgroup of $\mathcal{O}'(L)$ generated by all reflections. In other words, $\mathscr{R}(L)$ is the maximal reflection subgroup of $\mathcal{O}'(L)$. If $\mathscr{R}(L)$ has finite index in $\mathcal{O}'(L)$, then $L$ is called \textit{reflective}. As follows fromVinberg's lemma (see Lemma~\ref{lemma:Vinberg}), a Lorentzian lattice $L$ is reflective if and only if the fundamental Coxeter polyhedron for $\mathscr{R}(L)$ has finite volume in $\HH^d$.  

In 1972, Vinberg \cite{Vin72} suggested an algorithm allowing, at least theoretically, to construct subsequently the walls of a fundamental Coxeter polyhedron for any hyperbolic reflection group. In practice this algorithm is especially convenient if applied to a subgroup of type $\mathscr{R}(L)$, for some Lorentzian lattice $L$. Such convenience stems from the opportunity to enumerate all the reflections in  $\mathscr{R}(L)$ for a given lattice $L$; see Section~\ref{subsec:Lor} for details.

Let us turn to {\em thin hyperbolic reflection groups}. In \cite{Sarnak}, Sarnak provides an example of an infinite index {\em infinitely generated} thin reflection group produced by applying the Vinberg algorithm to a non--reflective Lorentzian lattice. Much more generally, we show that \textit{all} thin finitely generated reflection groups are, in a sense, coming from a similar construction.

It is known that if $L$ is not reflective but $\mathscr{R}(L)$ contains non--trivial reflections then $\mathscr{R}(L)$ has the fundamental Coxeter polyhedron $\mathscr{P}$ with infinitely many facets \cite{Vin72,Bottinelli}. Let $\mathscr{R}_m(L)$ be the reflection subgroup of $\mathscr{R}(L)$ generated by the first $m$ reflections produced by the Vinberg algorithm \cite{Vin72} and $\mathscr{P}_m$ be the respective Coxeter polyhedron.

The following immediately follows from Theorem~\ref{th:main}.

\begin{theorem}\label{th:thin}
Let $L$ be a non--reflective Lorentzian lattice over a totally real number field~$k$, and assume that $\mathscr{R}(L)$ is non--trivial. Then for sufficiently large $m$ all subgroups $\mathscr{R}_m(L)$ are thin in $\mathbf{PO}(L)_{\mathcal{O}_k}$. Moreover, they form a chain of subgroups where each inclusion has infinite index.  
\end{theorem}

\noindent Furthermore, we would like to stress the fact that the group $\mathscr{R}(L)$ in Theorem~\ref{th:thin} is non--trivial: there exists a Lorentzian lattice without any reflections in its isometry group at all (see Section~\ref{non-ref2}). 

It is also worth mentioning that recently Douba \cite{D22} showed that an irreducible right--angled Coxeter group with $n \geqslant 3$ generators embeds as a thin subgroup of a uniform arithmetic lattice in an indefinite orthogonal group $\mathrm{O}_{p,q}(\R)$ for some $p, q \ge 1$ satisfying $p + q = n$. Although such representations are largely out of scope in the context of discrete subgroups of $\Isom\,\HH^d$, they are related to our discussion in Remark~\ref{rem-1} and Remark~\ref{rem:vin-generalization}.

The following is a very rough (and elementary) converse of Theorem~\ref{th:thin}. 

\begin{theorem}\label{th:thin-2}
Every thin hyperbolic reflection group~$\Lambda < \Isom(\HH^d)$ is a subgroup of $\mathscr{R}_m(L) < \mathbf{PO}(L)_{\mathcal{O}_k}$, for some Lorentzian $\mathcal{O}_k$--lattice $L$ and natural ~$m \geqslant d+1$.
\end{theorem}

\noindent The proof of Theorem \ref{th:thin-2} follows in Section~\ref{sec:proofmain2}. 

\begin{remark}\label{compute}
Provided $L$ is any Lorentzian lattice (reflective and non--reflective alike), we obtain a thin subgroup $\Lambda$ of $\Gamma = \mathbf{PO}(L)_{\mathcal{O}_k}$ whenever the Vinberg algorithm produces a set of roots having connected Coxeter diagram that is neither elliptic nor affine, and does not correspond to a finite--volume Coxeter polyhedron (the latter can be checked by using Vinberg's criterion for hyperbolic Coxeter polyhedra \cite[Theorem 1]{Vin67}).
\end{remark}

\begin{remark}
    Based on Remark~\ref{compute}, we deduce that all thin reflection subgroups of $\Gamma = \mathbf{PO}(L)_{\mathcal{O}_k}$ are enumerable, given the existence of an algorithmic procedure to enumerate their generators. More precisely, there exists, at least in theory, an opportunity to list all Lorentzian lattices and subsequently apply our techniques to them. However, whether a given finitely generated subgroup $\Lambda < \Gamma$ can be identified as thin does not follow immediately from this procedure, as we still need to decide group isomorphism. On the other hand, if $\Lambda$ is given by its reflection generating set, then we will be able to detect it in finite time.

    It should be noted that thin groups are sometimes characterized as finitely generated Zariski dense infinite--index subgroups of lattices, regardless  of the latter being arithmetic or not. In this broader setting, the ambient lattices can have more complicated structures, precluding a precise Lorentzian lattice--based description as found in the arithmetic case. In particular, the lack of a comprehensive classification of non--arithmetic hyperbolic lattices complicates the prospect of even theoretically enumerating all thin groups as defined in this broader context.
\end{remark}

\subsection*{Funding}
N.B. was supported by the Russian Science Foundation, grant no. 21--41--09011.
A.K. was supported by the Swiss National Science Foundation, project PP00P2--202667.

\subsection*{Acknowledgements}

We thank Daniel Allcock, Ga\"el Collinet, and Sami Douba for numerous discussions and comments that helped improving this paper.

\section{Preliminaries}\label{sec:prelim}

\subsection{Lorentzian lattices and hyperbolic reflection groups}\label{subsec:Lor}

A {\em hyperbolic reflection group} is a discrete subgroup of $\Isom\,\HH^d$ generated by reflections in hyperplanes. Such groups are of some interest because of possibility to work explicitly with their fundamental domains, hyperbolic Coxeter polyhedra, i.e. hyperbolic convex polyhedra with dihedral angles being of the form $\pi/m$ for some $m \ge 2$. A good source of examples of hyperbolic reflection groups can be obtained from the class of arithmetic hyperbolic lattices of simplest type associated to the so--called Lorentzian lattices, and here we briefly describe this construction. 

Let $k$ be a totally real number field with the ring of integers $\mathcal{O}_k$. Then a freely generated $\mathcal{O}_k$--module $L$ equipped with a scalar product $(\cdot, \cdot)$ of signature $(d,1)$ is called a \textit{Lorentzian lattice} if for any non--identity embedding $\sigma: k \to \mathbb{R}$ the quadratic space $L\otimes_{\sigma(\mathcal{O}_k)}\mathbb{R}$ is positive definite.

Let us identify the space $L \otimes_{\mathcal{O}_k}\mathbb{R}$ with the real Minkowsky space $\mathbb{R}^{d,1}$, and let $\mathcal{O}'(L)$ be the group of linear transformations with $\mathcal{O}_k$--coefficients that preserve $L$ and map each connected component of the cone $\mathfrak{C} = \mathfrak{C}^+ \cup \mathfrak{C}^- = \{v \in \mathbb{R}^{d,1} \,|\, (v,v) < 0\}$ to itself. Then $\mathcal{O}'(L)$ is a discrete group of isometries of the hyperbolic $d$--space $\mathbb{H}^d = \{ v \in \mathbb{R}^{d,1} \cap \mathfrak{C}^+ \,|\, (v,v) = -1 \}$. 

It is known that $\Gamma = \mathbf{PO}(L)_{\mathcal{O}_k} \cong \mathcal{O}'(L)$ is an arithmetic lattice in $G = \PO(L)_\R \cong \Isom\,\HH^d$. Thus, the action of $\Gamma$ on $\HH^d$ has a finite--volume fundamental domain. Lattices in $G$ commensurable with $\Gamma$ as constructed above are called \textit{arithmetic hyperbolic lattices of simplest type}. This is the only type of arithmetic groups in $\Isom\,\HH^d$ containing reflections in hyperplanes. 

A primitive vector $e \in L$ is called a \textit{root} if $2 (e,x) \in (e,e) \cdot \mathcal{O}_k$ for all $x \in L$. Every root defines the associated \textit{reflection} $r_e(x) = x - \frac{2 (e,x)}{(e,e)} e$ acting on $L\otimes_{\mathcal{O}_k} \mathbb{R}$ that preserves the lattice $L$. The reflection $r_e$ induces a reflection in hyperplane $H_e = \{x \in \HH^d \mid (x,e) = 0\}$. Let $\mathscr{R}(L)$ be the subgroup of $\mathcal{O}'(L)$ generated by all such reflections. If $\mathscr{R}(L)$ has finite index in $\mathcal{O}'(L)$, then $L$ is called \textit{reflective}. 

The following fundamental fact \cite[Prop. 3]{Vin72} will be referred to as \textit{Vinberg's lemma}.

\begin{lemma}[\`{E}. B. Vinberg, \cite{Vin72}]\label{lemma:Vinberg}
Let $\Gamma$ be a discrete subgroup of $\Isom\, \HH^d$ and $W$ be its maximal reflection subgroup (i.e. generated by all reflections in $\Gamma$). Then $\Gamma$ decomposes into the semidirect product
$$
\Gamma = W \rtimes \mathrm{Sym}_\Gamma(\mathscr{P}),
$$
where $\mathscr{P}$ is the fundamental Coxeter polyhedron of $W$, and $\mathrm{Sym}_\Gamma(\mathscr{P}) < \Gamma$ is a subgroup of its symmetries belonging to $\Gamma$.  
\end{lemma}

Thus, by applying the above Lemma to $\Gamma = \mathbf{PO}(L)_{\mathcal{O}_k}$, a Lorentzian lattice $L$ is reflective if and only if the fundamental Coxeter polyhedron for $\mathscr{R}(L)$ has finite volume in $\HH^d$.

In 1972, Vinberg suggested \cite{Vin72} an algorithm allowing one, at least theoretically, to construct subsequently the walls of a fundamental Coxeter polyhedron for any hyperbolic reflection group. In practice this algorithm is especially convenient if applied to a subgroup of type $\mathscr{R}(L)$ for some Lorentzian lattice $L$. This convenience is simply explained by the opportunity to enumerate all the roots of a given lattice~$L$.

\subsection{Zariski dense discrete subgroups of $\Isom(\HH^d)$}

The following lemma is a well known characterization of discrete Zariski dense subgroups of $\Isom\,\HH^d$.

\begin{lemma}\label{lemma:Zariski-discrete}
    Let $\Gamma < \Isom\,\HH^d$ be a discrete group. It is Zariski dense if and only if its limit set $\Lambda_\Gamma$ is not contained in a hypersphere of $\partial \HH^d = \SSS^{d-1}$.
\end{lemma}
\begin{proof}
    Let $H = \mathrm{Zcl}(\Gamma)$ be the Zariski closure of $\Gamma$ in $\Isom\,\HH^d$. As follows from \cite[Part I, Ch. 5, Th. 3.7]{AVS93}, its connected component $H^\circ$ is either 
    \begin{itemize}
        \item a product $H_Y \times K_Y$ where $H_Y$ is a group of proper (orientation preserving) loxodromic elements with axes belonging to some plane $Y$ (with $0 \le \dim Y \le d$) and $K_Y$ is a compact Lie group consisting of all proper finite order elements fixing the plane $Y$; or
        \item some subgroup $H_p$ of the stabilizer (in $\Isom^+\,\HH^d$) of an ideal point $p \in \SSS^{d-1} = \partial \HH^d$.
    \end{itemize}
    Thus, if $\Gamma$ is Zariski dense, then $H^\circ$ can not be neither $H_p$ nor $H_Y \times K_Y$ where $Y$ is contained in some hyperplane. This implies that $\Lambda_\Gamma$ cannot be contained in any hypersphere of $\SSS^{d-1}$.

    Conversely, if $\Lambda_\Gamma$ is contained in a hypersphere $S = U \cap \SSS^{d-1}$ with $U$ being a hyperplane, then we conclude that $H^\circ$ is of the form $H_Y \times K_Y$ with $Y \subset U$, or $H^\circ = H_p$ with $p \in S$.
\end{proof}

The following is a geometric reformulation of the previous lemma applied to hyperbolic reflection groups.

\begin{lemma}\label{lemma:Zariski-reflection}
    Let $W$ be a finitely generated hyperbolic reflection groups. Suppose that $\P$ is its fundamental Coxeter polyhedron. Then $W$ is Zariski dense if and only if all of the following conditions hold:
    \begin{itemize}
        \item $W$ is not a spherical Coxeter group, i.e. $\P$ is not a fundamental cone with a finite vertex;
        \item $W$ is not a parabolic Coxeter group, i.e. $\P$ is not a fundamental cone with an ideal vertex;
        \item there is no hyperplane $U$ orthogonal to all of the walls of $\P$, i.e. $\P$ is non--degenerate.
    \end{itemize}
\end{lemma}
\begin{remark}
    As follows from \cite[Theorem 7]{Cor} (see \cite{BekH00, BenH04, Vin71} for the results used therein), every non--affine infinite irreducible Coxeter group $W$ is Zariski dense in the real algebraic group $\mathbf{O}(f)$, where $f$ is the quadratic form corresponding to the Tits representation. We would also conjecture that the same holds for all irreducible non--affine Tits--Vinberg representations as described in \cite{Vin71}.
\end{remark}

The next proposition was apparently known to Vinberg and other experts, however we did not manage to find it in the existing literature. A slightly less general case is written down in the PhD thesis of Bottinelli, see \cite[Lemma 7.1.2]{Bottinelli}.

\begin{proposition}\label{prop:inf-index-then-inf-sided}
    Let $\Gamma < \Isom(\HH^d)$ be a Zariski dense finitely generated discrete group containing at least one reflection, and let $W$ be the maximal reflection subgroup of $\Gamma$.
    Then $W$ is Zariski dense. Moreover, $W$ is finitely generated if and only if it is of finite index in $\Gamma$.
\end{proposition}
\begin{proof}
    First of all, recall that $W \trianglelefteq \Gamma$ and therefore by the well--known property of limit sets we have that $\Lambda_W = \Lambda_\Gamma$. Thus, by Lemma~\ref{lemma:Zariski-discrete}, the group $W$ is also Zariski dense.

    If $W$ is of infinite index in $\Gamma$, then by Vinberg's lemma~\ref{lemma:Vinberg} the symmetry group $\Sym_{\Gamma}(\P)$ is infinite. Now we use the fact that $W$ is Zariski dense. By Lemma \ref{lemma:Zariski-reflection}, all normal vectors of $\P$ span the entire Minkowski space $\R^{d,1}$. Obviously, $W$ is not a lattice, since it is of infinite index in another discrete group, and thus $\P$ is not a simplex and therefore has at least two disjoint facets. This implies that $\P$ is not finitely sided, since otherwise $\Sym_\Gamma (\P)$ would have a fixed point in $\HH^d$ (this fixed point would be the center of mass of all the midpoints of geodesic segments joining disjoint facets). The latter means $\Sym_\Gamma (\P)$ would be finite, a contradiction. That is, we just proved that $\P$ is infinite--sided, hence $W$ is infinitely generated.

    The converse is also true: a finite--index subgroup of a finitely generated group is finitely generated as a consequence of the Reidemeister--Schreier algorithm.
\end{proof}

\section{Algebraic proof of Theorem~\ref{th:main}}
\label{sec:alg-proof}

\begin{lemma}
    Let $G$ be an algebraic group, $B$ its subgroup, and $A \trianglelefteq B$. Then $B$ normalizes the Zariski closure $\Zcl(A)$ of $A$, and then $\Zcl(B)$ normalizes $\Zcl(A)$.
\end{lemma}
\begin{proof}
    The first claim is that the group $B$  normalizes the Zariski closure $\Zcl(A)$. Let $\Zcl(A) = C$. Consider the set $ D = \{ g \in G \mid b g b^{-1} \in C \textnormal{ for all } b \in B \}$. Then, $D$ is an algebraic subset in $G$ containing $A$, and thus it contains $\Zcl(A)$. By construction, $D$ is such a set that is conjugated by $B$ into $\Zcl(A)$, and therefore the first claim is proved.

    For the second claim, recall that the normalizer $N(D)$ of the algebraic set $D$ is an algebraic set itself. We just proved that $N(D) > B$ and thus $N(D) > \Zcl(B)$.
\end{proof}

\begin{corollary}\label{coro:normal-dense}
    If $B$ is Zariski dense in $G$, $A \trianglelefteq B$, then $\Zcl(A) \trianglelefteq G$.
\end{corollary}

The above corollary gives us another justification of the Zariski density of maximal reflection subgroups in Zariski dense discrete groups containing reflections.

\begin{corollary}\label{coro:W-dense}
    Let $\Gamma < \Isom(\HH^d)$ be a finitely generated Zariski dense discrete group containing at least one reflection. Then the maximal reflection subgroup $W < \Gamma$ is also Zariski dense.
\end{corollary}
\begin{proof}
    Since $W \triangleleft \Gamma$, then  Corollary~\ref{coro:normal-dense} and the fact that  $\OO_{d,1}(\R)$ is a simple real Lie group, imply that $W$ is Zariski dense.
\end{proof}

\begin{proof}[Algebraic proof of Theorem~\ref{th:main}]
    If $\Gamma < \Isom(\HH^d)$ is a finitely generated Zariski dense discrete group containing at least one reflection, then the maximal reflection subgroup $W$ of $\Gamma$, that is, generated by all reflections in $\Gamma$, is also Zariski dense in $\Isom(\HH^d)$ (see Proposition~\ref{prop:inf-index-then-inf-sided} or Corollary~\ref{coro:W-dense}). Moreover, Proposition~\ref{prop:inf-index-then-inf-sided} implies that if $W$ is of finite index in $\Gamma$, then it is finitely generated.
    
   Suppose now that $W$ has infinite index in $\Gamma$. Then $\Gamma$ contains a sequence of groups $W_n$ generated by first $n$ reflections assuming that we somehow enumerated the generators $\{r_n\}_{n=1}^{\infty}$ of $W$. For example, these first $n$ reflections may be produced by the Vinberg algorithm \cite{Vin72} applied to the group $W$.
   We claim that $W_n$ are eventually Zariski dense (finitely generated by construction) Coxeter groups such that $W_n$ is of infinite index in $W_{n+1}$ for every $n$. 

   Let us consider the tower of inclusions 
   $$
   \Zcl(W_1) < ... < \Zcl(W_n) < \Zcl(W_{n+1}) < ...
   $$
   As a sequence of algebraic groups, it stabilizes (one may consider the corresponding sequence of their defining ideals and recall that polynomial rings are Noetherian), meaning that for some $N > d$ we have $\Zcl(W_n)$, $n > N$, all equal to the same algebraic group. Moreover, the latter is exactly the Zariski closure $\Zcl(W)$ since every $r_n$ is included at some point in the above considered tower of inclusions. That $W_m$ is of infinite index in $W_{n}$ for any $m < n$, easily follows from \cite[Theorem 1.2]{felikson-tumarkin}: a finite index inclusion would compel the number of reflection generators of $W_m$ to be at least $n$.
\end{proof}

\begin{remark}\label{rem:vin-generalization}
    Our algebraic approach relies heavily on some facts specific to hyperbolic geometry: most notably, we use Vinberg's Lemma (Lemma~\ref{lemma:Vinberg}) and Proposition~\ref{prop:inf-index-then-inf-sided}, with Vinberg's lemma being particularly crucial. Should Vinberg's lemma be applicable to $\OO_{p,q}(\R)$ or $\mathrm{PU}_{n,1}(\R)$ contexts, it is likely that our main result (Theorem~\ref{th:main}) could be extended to a more general setting.

    More precisely, Zariski density of $W$ is demonstrated algebraically in a broader context, using only the normality of $W$ and the simplicity of the ambient Lie group, as shown in Corollary~\ref{coro:W-dense}. The criterion for $W$ being finitely generated (see Proposition~\ref{prop:inf-index-then-inf-sided}), depending on the finiteness of $W$'s index in $\Gamma$, hinges on both Vinberg's Lemma (Lemma~\ref{lemma:Vinberg}) and Lemma~\ref{lemma:Zariski-reflection}. The latter is a classical reinterpretation of Zariski density, which might also hold true in a more general setting.
\end{remark}

\section{Geometric proof of Theorem~\ref{th:main}}\label{sec:proofmain}

The starting point of this proof is the same as in Section \ref{sec:alg-proof}. If $\Gamma < \Isom\,\HH^d$ is a discrete Zariski dense group containing at least one reflection, then it contains the maximal (and non--trivial) reflection subgroup $W$. Moreover, by Vinberg's Lemma~\ref{lemma:Vinberg} we have that $\Gamma = W \rtimes \Sym_\Gamma(\P)$ where $\P$ is the fundamental Coxeter polyhedron for $W$. As we already know, $W$ is Zariski dense by Proposition~\ref{prop:inf-index-then-inf-sided}. The latter also implies that if $W$ has finite index in $\Gamma$ then $W$ is itself finitely generated.

From now on we assume that $W$ is an infinite--index subgroup in $\Gamma$. Then $\Sym_\Gamma(\P)$ is infinite and $\P$ has infinitely many facets, thus $W$ is infinitely generated by Proposition~\ref{prop:inf-index-then-inf-sided}. As in Section~\ref{sec:alg-proof}, let us denote by $W_n$ the reflection subgroup of $W$ generated by the first $n$ reflections produced by the Vinberg algorithm \cite{Vin72} applied to the group $W$, and let $\P_n$ be the fundamental Coxeter polyhedron of $W_n$.

First of all, let us observe that for $m \geqslant 2d-1$ the group $W_m$ is neither affine nor finite. Indeed, the only finite groups being simplex reflection groups with $d+1$ generators, while the maximum number of generators of an affine Coxeter group in $\HH^d$ is attained by a $(d-1)$--cube with $2d-2$ facets. The rest of the proof follows from Lemma~\ref{lemma:Zariski-reflection}: namely, it only remains to show that the Coxeter polyhedra $\P_n$ are eventually non--degenerate (that is, there is no hyperplane orthogonal to all the walls of $\P$).

\begin{lemma}\label{irreducible}
Only finitely many terms of the sequence $\{\P_n\}^\infty_{n=1}$ are degenerate Coxeter polyhedra.
\end{lemma}

\begin{proof}
We first claim that there are infinitely many non--degenerate Coxeter polyhedra among $\P_n$.

Suppose the contrary, and thus for some $m \geqslant 2d-1$ we have that every polyhedron $\P_n$, $n \geqslant m$ admits a hyperplane $U_n$ orthogonal to all of the facets of $\P_n$. 

Observe that adding the hyperplane $U_n$ to the facets of $\P_n$ gives a new non--trivial polyhedron $\widetilde{\P_n}$ in $\HH^d$, and its Gram matrix $G(\widetilde{\P_n})$ has signature\footnote{The {\em signature} of a real symmetric matrix $A$ is the triple $(p,q,r)$ of numbers of positive, negative, and zero eigenvalues of $A$, respectively.} $(d_n, 1, n-d_n-1)$, with $n \geqslant m > d_n+1$ and $d_n \leqslant d$.

Hence, all Gram matrices $G_n$ of $\P_n$ have signature $(d'_n,1,n-d'_n-1)$ with $d'_n  = d_n - 1 \geqslant d-1$ for all $n \to \infty$. Let us now consider the unit normal vectors defining the facets of $\P$. If they span the entire Minkowski space $\R^{d,1}$ then one could choose among them $d+1$ linearly independent vectors $r_{n_1}, \ldots, r_{n_s}$ with $s = d+1$ and indices $m \leqslant n_1 < \ldots n_s$. However, this would imply that $G_{n_s}$ has signature $(d,1,n_s-d-1)$ while its positive inertia index is less than $d$, as was discussed above. Thus, there is at least one hyperplane orthogonal to all of the walls of $\P$, but this contradicts Zariski density of $W$ by Proposition~\ref{prop:inf-index-then-inf-sided}. This proves our claim that it is impossible that only finitely many $\P_n$s are non--degenerate.

To conclude the proof of the lemma, we notice that if $\P_n$ is non--degenerate for some $n$ then obviously $\P_{n+1}$ also does not admit a hyperplane orthogonal to all of its walls, since otherwise it would hold for $\P_n$ either (recall that $\P_{n+1}$ is obtained from $\P_n$ just by adding one more wall). 
\end{proof}

We thus demonstrated that for some $m \geqslant 2d-1$, all $W_n$ with $n \geqslant m$ are Zariski dense. To complete the geometric proof of Theorem~\ref{th:main} we recall that by \cite[Theorem 1.2]{felikson-tumarkin} a finite index inclusion $W_m < W_n$, $m<n$, would compel the number of reflection generators of $W_m$ to be at least $n$, a contradiction.

\section{Proof of Theorem~\ref{th:thin-2}}\label{sec:proofmain2}

Let $\Lambda$ be a thin reflection subgroup contained in an arithmetic hyperbolic lattice $\Gamma < \mathbf{G}(\mathcal{O}_k)$. Then the lemma of Vinberg \cite[Lemma 7]{Vin67} which is, in the spirit of our paper, devoted to Zariski dense discrete groups containing reflections, implies that $\Gamma$ is a finite--index subgroup of $\mathbf{PO}(L)_{\mathcal{O}_k}$, for some Lorentzian lattice $L$ defined over a number field $k$ with the ring of integers $\mathcal{O}_k$. 

Then $\Lambda$ is a subgroup of $\mathscr{R}(L)$ and, moreover, $\Lambda$ is generated by a finite number of reflections in $\mathscr{R}(L)$. This implies that there exists a natural number $m \geq 3$ such that $\Lambda < \mathscr{R}_m(L)$. Notice that the mirrors of generating reflections of $\Lambda$ may belong to different chambers in the reflective tiling of $\HH^d$ under the $\mathscr{R}(L)$--action. Nevertheless, they are all conjugate in $\mathscr{R}(L)$ to a finite number of simple reflections bounding the fundamental Coxeter polyhedron $\P$ for the group $\mathscr{R}(L)$. Actually, we only need a finite number of simple reflections to write down the conjugating elements. Moreover, $\Lambda$ itself is a Coxeter group (for geometric reflection groups, see \cite[Part II, Ch. 5, Prop. 1.4]{AVS93}; for abstract Coxeter groups, see \cite[Theorem 3.3]{D}). Since $\Lambda$ is non--elementary, we have that $m \geqslant d+1$. 

\section{Examples}\label{sec:examples}

\subsection{Non--reflective lattice with roots}\label{non-ref1}

Let $L$ be the $\mathbb{Q}$--defined Lorentzian lattice associated to the form $f(x) = (x, x) = 3 x^2_0 + 14 x_0 x_1 + 98 x_0 x_2 + 49 x^2_2$, where $x = (x_0, x_1, x_2)$. Then $L$ is non--reflective. Indeed, the first $16$ roots of this form found by applying the Vinberg algorithm \cite{Bottinelli} are
\begin{align*}
     v_1 = (0, 7, -1), \, v_2 = (-7, -11, 2), \, v_3 = (0, 0, 1), \, v_4 = (-42, -24, 5), \\
     v_5 = (-98, 14, 1), \, v_6 = (-140, -31, 9), \, v_7 = (-168, -12, 7), \\ v_8 = (-21, -61, 14), \,
     v_9 = (-42, -94, 19), \, v_{10} = (-329, 22, 7),\\ v_{11} = (-42, -108, 23), \, v_{12} = (-252, -74, 19), \,
     v_{13} = (-273, -37, 14), \\ v_{14} = (-28, -86, 21), \, v_{15} = (-56, -151, 33), \, v_{16} = (-49, -154, 39).
\end{align*}

The involution $v_5 \to v_{10},\, v_{14} \to v_{16}$ turns out to be an infinite--order symmetry \cite[\S 7.3.1]{Bottinelli} of the reflective part of $L$.  

The Coxeter scheme of the root system above quickly becomes complicated (it becomes non--planar for more than $4$ roots), so we only draw it for the roots $v_1, \ldots, v_4$ (Figure~\ref{fig:thin_scheme1}) which already gives a sequence of nested thin subgroups in $\Gamma = \mathbf{PO}(L)_{\mathbb{Z}}$. Below we also list the first $4$ reflections $r_i$, $i=1,2,3,4$, each corresponding to its root $v_i$, $i=1,2,3,4$.

\begin{equation*}\label{gens:1}
r_1 = \left(\begin{array}{rrr}
1 & 0 & 0 \\
0 & 1 & 14 \\
0 & 0 & -1
\end{array}\right), \quad
r_2 = \left(\begin{array}{rrr}
1 & -14 & -70 \\
0 & -21 & -110 \\
0 & 4 & 21
\end{array}\right), 
\end{equation*}
\begin{equation*}\label{gens:2}
r_3 = \left(\begin{array}{rrr}
1 & 0 & 0 \\
0 & 1 & 0 \\
-2 & 0 & -1
\end{array}\right), \quad
r_4 = \left(\begin{array}{rrr}
-83 & -504 & -3108 \\
-48 & -287 & -1776 \\
10 & 60 & 371
\end{array}\right).
\end{equation*}

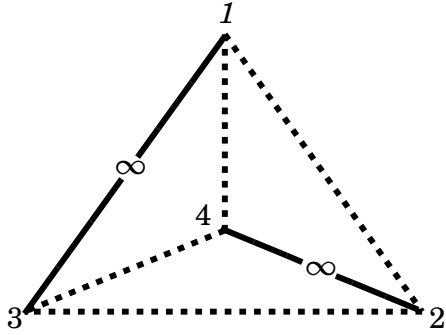
\begin{figure}
\centering
\scalebox{1.00}{
\tikzset{every picture/.style={line width=0.75pt}} 

\begin{tikzpicture}[x=0.75pt,y=0.75pt,yscale=-1,xscale=1]

\draw [line width=2.25]  [dash pattern={on 2.53pt off 3.02pt}]  (300.5,101) -- (400,241) ;
\draw [line width=2.25]    (300.5,102) -- (283.67,125.33) ;
\draw [line width=2.25]    (283.67,125.33) -- (266.83,148.67) ;
\draw [line width=2.25]    (266.83,148.67) -- (256.5,163) ;
\draw [line width=2.25]    (248.5,174) -- (233.17,195.33) ;
\draw [line width=2.25]    (233.17,195.33) -- (216.33,218.67) ;
\draw [line width=2.25]    (216.33,218.67) -- (199.5,242) ;
\draw [line width=2.25]  [dash pattern={on 2.53pt off 3.02pt}]  (400,241) -- (199.5,241) ;
\draw [line width=2.25]  [dash pattern={on 2.53pt off 3.02pt}]  (300.5,101) -- (300.5,200) ;
\draw [line width=2.25]  [dash pattern={on 2.53pt off 3.02pt}]  (300.5,200) -- (199.5,241) ;
\draw [line width=2.25]    (298.5,199) -- (315.08,205.83) ;
\draw [line width=2.25]    (315.08,205.83) -- (331.67,212.67) ;
\draw [line width=2.25]    (331.67,212.67) -- (339.5,216) ;
\draw [line width=2.25]    (356.5,223) -- (364.83,226.33) ;
\draw [line width=2.25]    (364.83,226.33) -- (381.42,233.17) ;
\draw [line width=2.25]    (381.42,233.17) -- (398,240) ;

\draw (296,83.4) node [anchor=north west][inner sep=0.75pt]    {$\mathit{1}$};
\draw (404,237.4) node [anchor=north west][inner sep=0.75pt]    {$2$};
\draw (187,237.4) node [anchor=north west][inner sep=0.75pt]    {$3$};
\draw (284,185.4) node [anchor=north west][inner sep=0.75pt]    {$4$};
\draw (244,163.4) node [anchor=north west][inner sep=0.75pt]    {$\infty $};
\draw (340,214.4) node [anchor=north west][inner sep=0.75pt]    {$\infty $};

  \fill[black] (300.5,102) circle (3pt) ;
  \fill[black] (400,241) circle (3pt) ; 
  \fill[black] (199.5,241) circle (3pt) ;
  \fill[black] (300.5,200) circle (3pt) ;

\end{tikzpicture}
}
\caption{The Coxeter scheme for $v_1, \ldots, v_4$ (Section \ref{non-ref1})}\label{fig:thin_scheme1}
\end{figure}

\subsection{A lattice without roots}\label{non-ref2}

It is important to ask that the lattice $L$ in Theorem~\ref{th:thin} have roots. A curious example of the contrary is the $\mathbb{Q}$--defined Lorentzian lattice $N$ associated to the form $f(x) = (x,x) =  49 x_1^2 + 98 x_0 x_2 + 14 x_1 x_2 + 3 x_2^2$. It turns out that $N$ has no roots. This instance was communicated to the authors by Ga\"el Collinet \cite{Collinet}.  

In order to show it, let us recall that by a result of Vinberg \cite{Vin84} once we have a root $r$ then $f(r)$ divides twice the last invariant factor of the Gram matrix of $f$. Thus we need first to decide which divisors of $4802$ can be integrally represented by $f$. Modulo local obstructions only $49$, $98$, $2401$, and $4802$ remain (and all of them are indeed integrally represented by $f$). 

Let us consider the case of a root $r = (k_1, k_2, k_3) \in \mathbb{Z}^3$ with $f(r) = 49$. The crystallographic conditions read as $2 (r, e_i) \in f(r) \cdot \mathbb{Z}$, $i=1,2,3$, where $e_i$ is the $i$--th basis vector spanning $N$. Moreover, $k_i$'s are assumed to be coprime. It follows from the crystallographic conditions that we have $k_1 = m_1$, $k_2 = 7 m_2 - 3 m_3$, and $k_3 = 7 m_3$, for $m_1, m_2, m_3 \in \mathbb{Z}$. 

Hence $q(r) = f(r)/49 = 49 m_2^2 + 14 m_1 m_3 - 28 m_2 m_3 + 6 m_3^2$. However, $q$ does not integrally represent $1$. Analogous arguments exclude the remaining possible roots lengths $98$, $2401$, and $4802$.

\bibliographystyle{siam}
\bibliography{biblio.bib}

\end{document}